\documentclass[a4paper]{amsart}
\usepackage{a4}
\usepackage{amsthm,enumerate,amssymb,hyperref,color}

\DeclareMathOperator{\kew}{Skew}
\DeclareMathOperator{\ym}{Sym}

\DeclareMathOperator{\Span}{span}

\begin{document}

\newcommand{\A}{\cal{A}}
\newcommand{\F}{\cal{F}}
\newcommand{\R}{\mathbb{R}}
\newcommand{\cH}{\cal{H}}
\newcommand{\T}{\cal{T}}
\newcommand{\Z}{\cal{Z}}
\newcommand{\J}{\cal{J}}
\newcommand{\cS}{\cal{S}}
\newcommand{\M}{\cal{M}}
\newcommand{\K}{\cal{K}}
\newcommand{\C}{\mathbb{C}}
\newcommand{\U}{\cal{U}}
\newcommand{\B}{\cal{B}}
\newcommand{\V}{\cal{V}}
\newcommand{\cL}{\cal{L}}
\newcommand{\W}{\cal{W}}
\newcommand{\FF}{\mathbb F}
\newcommand{\KK}{\mathbb F_0}
\newcommand{\Q}{\mathbb Q}
\newcommand{\N}{\mathbb N}
\newcommand{\wh}[1]{\widehat{#1}}
\newcommand{\rsl}{\mathfrak{sl}}
\newcommand{\wt}{\widetilde}
\newcommand{\ov}{\overline}
\newcommand{\kh}{\mathcal{K}(\mathcal{H})}

\newcommand{\Skew}{{\kew\FF\langle\bar X,\bar X^*\rangle}}
\newcommand{\Sym}{{\ym\FF\langle\bar X,\bar X^*\rangle}}

\newcommand{\ax}{\langle\bar X\rangle}

\newtheorem{theorem}{Theorem}[section]
\newtheorem{proposition}[theorem]{Proposition}
\newtheorem{lemma}[theorem]{Lemma}
\newtheorem{corollary}[theorem]{Corollary}

\theoremstyle{definition}
\newtheorem{remark}[theorem]{Remark}
\newtheorem{definition}[theorem]{Definition}
\newtheorem{example}[theorem]{Example}
\newtheorem{examples}[theorem]{Examples}
\newtheorem{question}[theorem]{Question}

\newcommand\cal{\mathcal}

\linespread{1.05}

\def\bh{\mathcal B(\mathcal H)}
\title[Values of Noncommutative Polynomials]{A Note on Values of Noncommutative Polynomials}

\author[Matej Bre\v sar and Igor Klep]{Matej Bre\v sar${}^1$ and Igor Klep${}^2$}

\address{ Faculty of Mathematics and Physics, University
of Ljubljana, and \\
  Faculty of Natural Sciences and Mathematics, University of Maribor, Slovenia	}
\email{matej.bresar@fmf.uni-lj.si}
\email{igor.klep@fmf.uni-lj.si}
\thanks{${}^1$Supported by the Slovenian Research Agency (program No. P1-0288).}
\thanks{${}^2$Supported by the Slovenian Research Agency (program No. P1-0222).}
\subjclass[2000]{08B20, 16R99, 47L30}
\date{09 September 2009}
\keywords{Noncommutative polynomial, Lie ideal, Hilbert space, bounded operator, compact operator}

\begin{abstract}
We find a class of algebras $\A$ satisfying the following property: 
for every nontrivial
noncommutative polynomial $f(X_1,\ldots,X_n)$, the linear span
of all its values $f(a_1,\ldots,a_n)$, $a_i\in \A$, 
equals $ \A$. This class includes the algebras 
of all bounded and  all compact
operators  on an infinite dimensional  Hilbert space.
\end{abstract}

\maketitle
\section{Introduction}

Starting with Helton's seminal paper \cite{Hel} there has been considerable interest  over the last years  in values of noncommuting polynomials on matrix algebras. In one of the papers in this area 
the second author and Schweighofer \cite{KS} showed  that  Connes' embedding conjecture is  equivalent  to a certain algebraic assertion which involves the trace of polynomial values on matrices.     
This has motivated us \cite{BK} to consider the linear span of values of a noncommutative polynomial $f$ on the matrix algebra $M_d(\FF)$; here, $\FF$ is a field with char$(\FF) =0$. It turns out \cite[Theorem 4.5]{BK} that this span can be either:
\begin{itemize}
\item[(1)]
$\{0\}$;
\item[(2)] the set of all scalar matrices;
\item[(3)] the set of all trace zero matrices; or 
\item[(4)] the whole algebra $M_d(\FF)$.
\end{itemize}
\noindent From the precise statement of this theorem it also follows that if 
$2 d> \deg f$, then (1) and (2) do not
occur
and (3) occurs only when $f$ is a sum of commutators.

What to except in infinite dimensional analogues of  $M_d(\FF)$? More specifically, let  $\cH$ be infinite dimensional Hilbert space, and let $\bh$ and $\kh$ denote the algebras of 
all bounded and compact linear operators on $\cH$, respectively. What is 
 the linear span of polynomial values in $\bh$ and $\kh$? 
 A very special (but decisive, as we shall see) case of this question was settled by Halmos \cite{Hal} and Pearcy and Topping \cite{PT} (see also Anderson \cite{And})
 a long time ago:  every operator in $\bh$ and $\kh$, respectively, is a sum of commutators. That is, the linear span of values of the polynomial
$X_1X_2-X_2X_1$ on $\bh$ and $\kh$ is the whole $\bh$ and $\kh$, respectively.
We will prove that the same is true for \emph{every} nonconstant polynomial. This result will be derived as a corollary of our main theorem which presents a class of algebras with the property that the span of values of ``almost" every polynomial is equal to the whole algebra.

\section{Results}

By $\FF\langle \bar{X}\rangle$ we denote the free algebra over a field $\FF$ generated by $\bar{X} = \{X_1,X_2,\ldots\}$, i.e., the algebra of all noncommutative polynomials in $X_i$. Let $f = f(X_1,\ldots,X_n)\in \FF\langle \bar{X}\rangle$. We say that $f$ is {\em homogeneous} in the variable $X_i$ if all monomials of $f$ have the same degree in $X_i$. If this degree is $1$, then we say that $f$ is {\em linear} in $X_1$. If $f$ is linear in every variable $X_i$, $1\leq i \le n$, then we say that $f$ is {\em multilinear}.

Let  $\A$ be an algebra over $\FF$.
By $f(\A)$ we denote the set of all values $f(a_1,\ldots,a_n)$ with $a_i\in \A$, $i=1,\ldots,n$. Recall that  $f = f(X_1,\ldots,X_n)\in \FF\langle \bar{X}\rangle$ is said to be an {\em identity} of $\A$ if $f(\A)=\{0\}$. If $f(\A)$ is contained in the center of $\A$, but $f$ is not an identity of $\A$,
then $f$ is said to be a {\em central polynomial} of $\A$.    By $\Span f(\A)$ we denote the linear span of $f(\A)$. We are interested in the question when does $\Span f(\A)=\A$ hold.

For the proof of our main theorem 
three rather  elementary lemmas will be needed.
The first and the simplest one is a slightly simplified version of \cite[Lemma 2.2]{BK}. Its proof is based on the standard Vandermonde argument.

\begin{lemma} \label{R1}
Let $\V$ be a vector space over an infinite field $\FF$, and let $\U$ be a subspace. Suppose that $c_0,c_1,\ldots,c_n\in \V$ are such that
$
\sum_{i=0}^n \lambda^i c_i \in \U
$
for all $\lambda\in\FF$. Then each $c_i\in \U$.
\end{lemma}

Recall that a vector subspace $\cL$ of $\A$ is said to be a {\em Lie ideal} of $\A$ if $[\ell,a]\in \cL$ for all
$\ell\in \cL$ and $a\in \A$; here, $[u,v] = uv - vu$.    For a recent treatise of Lie ideals from an algebraic
as well as functional analytic viewpoint we refer the reader to 
\cite{BKS}.

 Our second lemma is a special case of \cite[Theorem 2.3]{BK}.

\begin{lemma} \label{L1}
Let $\A$ be an algebra over  an infinite field  $\FF$, and let $f  \in \FF\langle \bar{X}\rangle$. Then $\Span f(\A)$ is a  Lie ideal of $\A$.
\end{lemma}

Every vector subspace of the center of $\A$ is obviously a Lie ideal of $\A$. Lie ideals that are not contained in the center are called {\em noncentral}. The third lemma follows from an old  result of Herstein \cite[Theorem 1.2]{Her}.

\begin{lemma} \label{L2}
Let $\cS$ be a simple algebra over  a field  $\FF$ with {\rm char}$(\FF)\ne 2$. If $\M$ is both a noncentral  Lie ideal of $\cS$ and a subalgebra of $\cS$, then $\M = \cS$.
\end{lemma}

We are now in a position to prove our main result.

\begin{theorem} \label{T}
Let $\cS$ and $\B$ be algebras over a field  $\FF$ with {\rm char}$(\FF)=0$, and let $\A = \cS\otimes \B$. Suppose that $\cS$ is simple, and suppose that $\B$ satisfies 
\begin{itemize} \item[(a)] every element in $\B$ is a sum of commutators; and
\item[(b)] for each $n\ge 1$, every  element in $\B$ is a linear combination of elements $b^n$, $b\in\B$.
\end{itemize} If $f\in \FF\langle \bar{X}\rangle$ is neither an identity nor a central polynomial of $\cS$, then $$\Span f(\A)=\A.$$ 
\end{theorem}

\begin{proof}
Let $f= f(X_1,\ldots,X_n)$. Let us write $f= g_i + h_i$ where $g_i$ is a sum of all monomials of $f$ in which $X_i$ appears and  $h_i$ is a sum of all monomials of $f$ in which $X_i$ does not appear. Thus, $h_i= h_i(X_1,\ldots,X_{i-1},X_{i+1},\ldots,X_n)$ and hence 
$$h_i(a_1,\ldots,a_{i-1},a_{i+1},\ldots,a_n) = f(a_1,\ldots,a_{i-1},0,a_{i+1},\ldots,a_n)$$ 
for all $a_i\in\A$. Therefore $\Span h_i(\A)\subseteq \Span f(\A)$, which clearly implies $\Span g_i(\A)\subseteq \Span f(\A)$. At least one of $g_i$ and $h_i$ is neither an identity nor a central polynomial of $\cS$. Therefore there is no loss of generaliy in assuming that either $X_i$ appears in every monomial of  $f$  or  $f$ does not involve $X_i$ at all. Since $f$ cannot be a constant polynomial and hence it must involve some of the $X_i$'s,  we may assume, again without loss of generality, that each monomial of $f$ involves all $X_i$, $i=1,\ldots,n$.  

Next we claim  that there is no loss of generality in assuming that $f$ is homogeneous in $X_1$.
Write $f=f_1+\ldots +f_m$, where  $f_i$ is the sum of all monomials of $f$ that have degree $i$ in $X_1$.
Note that
$$f(\lambda a_1,a_2,\ldots,a_n) =  \sum_{i=1}^m \lambda^i f_i(a_1,\ldots,a_n) \in \Span f(\A)$$
 for all $\lambda\in \FF$ and all $a_i\in \A$, so $ f_i(a_1,\ldots,a_n) \in \Span f(\A)$ by Lemma \ref{R1}. Thus, $\Span f_i(\A)\subseteq \Span f(\A)$. At least one $f_i$  is neither an identity nor a central polynomial of $\cS$. Therefore it suffices to prove the theorem for $f_i$. This proves our claim.
 
 Let us now show that there is no loss of generality in assuming that $f$ is linear in $X_1$.
 If $\deg_{X_1}f>1$, we apply the multilinearization process to $f$, i.e., we introduce a new polynomial $ \Delta_{1,n+1}f = f'(X_1,\ldots,X_n,X_{n+1})$: 
$$
f'=f\left(X_1+X_{n+1},X_2,\dots,X_n\right)-
 f\left(X_1,X_2,\dots,X_n\right)
 -f\left(X_{n+1},X_2,\dots,X_n\right).
$$
This reduces the degree in $X_1$ by one.
Clearly, $\Span f'(\A)\subseteq\Span f(\A)$. 
Observe that $f$ can be retrieved from $f'$ by 
resubstituting $X_{n+1}\mapsto X_1$, more exactly
$$
(2^{\deg_{X_1}f}-2) f= f'(X_1,\ldots,X_n,X_1).
$$
Hence $f'$ is not an identity nor a central polynomial of $\cS$.
Note however that $f'$ is not necessarily homogeneous in $X_1$, but for all
its homogeneous components $f'_i$ we have
$\Span f'_i(\A)\subseteq \Span f'(\A)$;
one can check this by
using Lemma \ref{R1}, like in the previous paragraph. 
At least one of these components, say $f_j'$,
is not an identity nor a central polynomial of $\cS$.
Thus we restrict our attention to $f_j'$.
 If necessary, we continue applying
 $\Delta_{1,\text{\textvisiblespace}}$, and after a finite number of 
 steps we obtain a polynomial $\Delta f$ linear in $X_1$,
which is neither an identity nor a central polynomial of $\cS$, and
 satisfies
 $\Span \Delta f(\A)\subseteq \Span f(\A)$. Hence we may assume
 $f$ is linear in $X_1$.
 
 Repeating the same argument with respect to other variables we finally see that without loss of generality we may assume that $f$ is multilinear.

 Set $\cL = \Span f(\A)$ and 
 $\M= \{m\in\cS\,|\, m \otimes \B\subseteq \cL\}$. By Lemma \ref{L1}, $\cL$ is a Lie ideal of $\A$. Therefore $[m,s]\otimes b^2 = [m\otimes b,s\otimes b] \in \cL$ for all $m\in\M$, $b\in \B$, $s\in\cS$. Using (b) it follows that $[m,s] \in\M$. Therefore
 $\M$ is a Lie ideal of $\cS$.
Pick $s_1,\ldots,s_n\in \cS$ such that $s_0 = f(s_1,\ldots,s_n)$ does not lie in the center of $\cS$. For every $b\in \B$ we have
$$
 s_0\otimes b^n = f(s_1\otimes b, s_2\otimes b,\ldots,s_n\otimes b) \in \cL.
$$ 
In view of (b) this yields $s_0\in \M$. Accordingly, $\M$ is a noncentral Lie ideal of $\cS$. Next, given $m\in \M$ and $b,b'\in\B$, we have 
 $$
 m^2 \otimes [b,b'] = [m\otimes b, m\otimes b'] \in\cL.
 $$
By (a), this gives $m^2\in\M$. From
  $$m_1 m_2 = \frac{1}{2}([m_1,m_2] + (m_1 + m_2)^2 - m_1^2 - m_2^2)$$
  it now follows that $\M$ is a subalgebra of $\cS$. Using Lemma \ref{L2} we now conclude that $\M = \cS$, i.e., $\A=\cS \otimes \B \subseteq \cL \subseteq\A$.
  \end{proof}
  
It is easy to see that (b) is fulfilled if $\B$ has a unity. In this case the proof can be actually slightly simplified by avoiding involving powers of elements in $\B$. Further, every $C^*$-algebra
$\B$ satisfies (b). Indeed, every element in $\B$ is a linear combination of positive elements, and for positive elements we can define  $n$th roots.

\begin{corollary}\label{thm:halmos}
Let $\cH$ be an infinite dimensional Hilbert space. Then 
$$\Span f(\bh)=\bh\quad\text{and}\quad\Span f(\kh)=\kh
$$ for every nonconstant polynomial
 $f\in\C\langle\bar X\rangle$.
\end{corollary}

\begin{proof}
It is well known that there does not exist a nonzero polynomial that was an identity of $M_d(\C)$ for every $d\ge 1$, cf.~\cite[Lemma 1.4.3]{Row}. 
Therefore there exists $d\ge 1$ such that $[f,X_{n+1}]$ is not an identity of 
$M_d(\C)$. This means that $f$ is neither an identity nor a central polynomial of $M_d(\C)$. Since $\cH$ is infinite dimensional, we have $\bh\cong M_d(\bh)\cong M_d(\C) \otimes \bh$, and similarly
$\kh\cong M_d(\C) \otimes \kh$. Now we are in a position to use  Theorem \ref{T}. Indeed,
$M_d(\C)$ is a simple algebra, and the algebras $\bh$ and $\kh$ satisfy (a) by \cite{Hal} and \cite{PT}, and they satisfy (b) by the remark preceding the statement of the corollary.
\end{proof}

\end{document}